\documentclass[11pt]{amsart}

\usepackage{amssymb,latexsym,amsmath,amsthm,epsf}
\usepackage{latexsym}
\usepackage{amscd}

\newtheorem{thm}{Theorem}

\newtheorem{ob}[thm]{Observation}
\newtheorem{conj}{Conjecture}

\newcommand{\gcol}{{\rm gcol}}

\begin{document}
\title{Total game coloring of graphs}
\author{T. Bartnicki}
\address{Faculty of Mathematics, Computer Science, and Econometrics,
University of Zielona G\'{o}ra, 65-516 Zielona G\'{o}ra, Poland}
\email{T.Bartnicki@wmie.uz.zgora.pl }
\author{Z. Miechowicz}
\address{Faculty of Mathematics, Computer Science, and Econometrics,
University of Zielona G\'{o}ra, 65-516 Zielona G\'{o}ra, Poland}
\email{z.miechowicz@wmie.uz.zgora.pl}

\begin{abstract}
Total variant of well known graph coloring game is considered. We determine exact values of $\chi''_{g}$ for some classes of graphs and show that total game chromatic number is bounded from above by $\Delta+3\Delta^+$. We also show relation between total game coloring number and game coloring index.
\end{abstract}

\maketitle

\section{Introduction}

Graph coloring game was first introduced in 1981 \cite{gar} by Steven Brams, but only after ten years graph game parameters were first time considered in mathematical paper \cite{bod} by Bodlaender. Various variants of graph coloring game have been studied until now. Profoundly explored is the original version of the game, in which players alternately color vertices of given graph (see survey \cite{bagr-06}), but there is also a lot of results about game chromatic index, the parameter connected with game in which players color edges of a graph (\cite{Anders},\cite{BG},\cite{fkkt}). Natural extension of this research is to consider a game in which players color both, vertices and edges of a graph.

Let a graph $G$ and a set of colors $C$ be given. Total graph coloring game is a game with the following rules:

\begin{itemize}
   \item Two players, Alice and Bob alternately color vertices and edges of $G$, using colors from $C$ with Alice playing first,
   \item both players have to respect proper coloring rule: in each moment of the game each pair of incident objects in graph $G$ (i.e. vertex-vertex, vertex-edge, edge-edge) have to receive different colors,
   \item Alice wins if whole graph has been colored properly with available colors and Bob otherwise.
\end{itemize}

With this game we can relate a natural parameter. The {\it total game chromatic number}, denoted $\chi''_{g}$, is the least number of colors in $C$ for which Alice has a winning strategy.

In section 2. we present some basic properties of $\chi''_{g}$ and exact results. In section 3. we give a strategy for Alice to win the game with $\Delta+3\Delta^+$ colors. To show the strategy we define total game coloring number and show a strategy for Alice in related marking game. We also show that $\gcol'(G)\leq \gcol''(G)$.

\section{Basic observations}

We start with trivial bounds for the total game chromatic number and give two simple examples, showing that the  total game chromatic number can
be larger than the usual  total chromatic number, and that $\chi''_{g}$ is not monotone on taking subgraphs.

Upper bound for $\chi''_{g}$ is given by counting neighbors of edges and vertices. Since every vertex, or edge in a graph $G$ have at most $2\Delta(G)$ incident objects we have the following inequality:

\begin{ob}
$$\chi''_{g}(G)\leq 2\Delta(G)+1$$
\end{ob}

Typical property of this type graph game parameters is, that they are not monotone on taking subgraphs. It is true also for total game chromatic number. We can see it analyzing these two examples.

\begin{ob}
$$\chi''_{g}(K_3)=5$$
\end{ob}

\begin{proof}
From the trivial upper bound we know that $5$ colors is enough for Alice to win the game on $K_3$. We will show the strategy for Bob in the game with $4$ colors.

In her first move Alice have to color one of the vertices or edges using first color. Bob can answer by coloring unique object in $K_3$, which is nonincident with the one chosen earlier by Alice, using second color. Now Alice have to use third color, and Bob answers in the same way- he colors the only object nonincident with the one colored by Alice in her last move using the last available color. Now we have only one uncolored edge, and one uncolored vertex left and they both have neighbors in each of available colors.
\end{proof}

\begin{ob}
$$\chi''_{g}(K_3\cup K_1)=3$$
\end{ob}

\begin{proof}
We know that $\chi''(K_3\cup K_1)=3$. We will show the strategy for Alice in the game on $K_3\cup K_1$ with $3$ colors.

In her first move Alice colors isolated vertex with any color. Now Bob has to start game on $K_3$ and Alice have good answer for every Bob's move. She colors the only object nonincident with the one colored by Bob in his last move with the same color, so the whole graph can be colored properly with $3$ colors.
\end{proof}

This anomaly causes, that induction methods can not be used to explore total graph coloring game. But we can apply well known activation strategy.

\section{Total game coloring number}

A very useful tool in finding upper bounds for game parameters is the activation strategy. To use it we define auxilary game - total marking game. Let $k\in N$ and a graph $G$ be given. {\it Total marking game} with parameter $k$ is a game with following rules:
\begin{itemize}
   \item Two players Alice and Bob alternately mark vertices or edges of $G$ with Alice playing first,
   \item Alice wins if, in each part of the game each unmarked vertex and each unmarked edge in $G$ have at most $k-1$ marked neighbors, otherwise Bob wins.
\end{itemize}

We can define parameter concerned with that game. {\it Total game coloring number} of graph $G$, denoted $\gcol''(G)$, is the least $k$, such that Alice has a winning strategy in total marking game on $G$ with parameter $k$.

It is easy to see, that if Alice has a winning strategy in total marking game on $G$ with parameter $k$, she has also a winning strategy in total coloring game on $G$ with a set of $k$ colors. So, for every graph $G$, we have $\chi_g''(G)\leq \gcol''(G)$.

Analogous marking game was defined for edge coloring and the parameter related with it is game chromatic index denoted by $\gcol'(G)$. An expected inequality, that connects $\gcol'(G)$ and $\gcol''(G)$, holds.

\begin{thm}
Let G be a graph. Then $\gcol'(G)\leq \gcol''(G)$.
\end{thm}

\begin{proof}
Let $G$ be a graph with the set of vertices $V=\{v_1,...,v_n\}$ and edges $E=\{e_1,...,e_m\}$. Assume that Bob has a winning strategy in edge marking game with parameter $k$ on $G$. We will show the winning strategy for Bob in total marking game with the same parameter on $G$.

If Bob can win in edge marking game on $G$, he can also win marking game on linegraph $L(G)$ with the same parameter. Hence this game is monotone on taking subgraphs \cite{wuzh-03} Bob can also win marking game with the same parameter on $L(G)\cup nK_1$, where $nK_n$ is $n$ isolated vertices. Bob plays total marking game according to the wining strategy at marking game on  $L(G)\cup nK_1$. Bob, before the game, will mark every vertex of $L(G)$ as corresponding edge in $G$, and remaining vertices of $L(G)\cup nK_1$ with $v_1,...,v_n$. His strategy is the following. Bob will simulate every Alice's move on $L(G)\cup nK_1$ by marking object with the same label, as she chose in $G$. His winning strategy on  $L(G)\cup nK_1$ make him chose vertex $v_i$, or edge $e_j$, and in his move in a game on $G$ he will chose an object with exactly the same label. Hence he has a winning strategy in marking game on $L(G)\cup nK_1$, there must be a moment in the game, when some of the vertices of $L(G)$, say $e_i$, has more than $k-1$ marked neighbors. But in the same moment the edge $e_i$ in $G$ has at least $k$ marked neighbors in $G$.
\end{proof}

Modification of the activation strategy from edge marking game \cite{BG} will give a bound of $\gcol''(G)$ in terms of $\Delta$ and maximum outdegree of graph. For a vertex $x$ of directed graph $D$ we denote the set of its out-edges by $E^+(x)$ and the set of its in-edges by $E^-(x)$.

\begin{thm}
Let $G$ be a graph which can be oriented in such a way, that it's out degree is at most $k$.
Then $\gcol''(G)\leq \Delta (G)+3k+1$.
\end{thm}

\begin{proof}
Let $D$ be a directed graph obtained by orienting the edges of $G$ such that $%
\Delta ^{+}(D)\leq k$. We shall describe a strategy for Alice
guaranteeing that at any moment of the game, each unmarked
edge or vertex has at most $\Delta (G)+3k$ marked neighbors, so Alice can always win the marking game on $G$ with parameter $\Delta (G)+3k+1$ .

During the game Alice (and only she) will activate vertices and edges of $D$. The edge, or the vertex once activated will be active till the end of the game. Bob can see which edges or vertices are active, but
this information does not help him. Before Alice marks an edge or vertex she will ,,jump'' on the edges and vertices of $D$ and activate it according to the following rules:

\begin{enumerate}
	\item Alice always jumps according to the orientation of $D$,
	\item in her first move Alice starts jumping from any edge,
	\item after Bob's move Alice jumps to the object he has just marked, activates it and starts her jumping procedure from it,
	\item from an oriented edge $\vec{xy}$ Alice jumps to $y$, unless $y$ is already marked. In that case Alice jumps to the unmarked edge from the set $E^+(y)$,
	\item from $x\in V(D)$ Alice jumps to any unmarked edge from the set $E^+(x)$,
	\item if Alice jumps to inactive object, she activates it. If Alice jumps to active object, she marks it and stops,
	\item if Alice can't make first jump, she starts jumping from any edge,
	\item if Alice can't make further jump, she marks last visited and activated object.
\end{enumerate}

Note that after Alice's move each marked object is also active. After Bob's move there is at most one more marked inactive object. For any unmarked object we will count the maximal number of active objects incident with it after Bob's move.

Let $x\in V(D)$ be an active and unmarked vertex. It is incident with at most $\Delta(G)$ vertices, and we assume, that they all can be active. Vertex $x$ has at most $k$ out-edges, and they all can be active. Only two among its in-edges can be active. One from which Alice jumped to $x$, and one just marked by Bob. So unmarked vertex $x$ can have at most $\Delta+ k+2$ active neighbors.

Let $\vec{xy}\in E(D)$ be unmarked and active edge. It has at most $\Delta -1$ incident edges, which are incident also with $y$, and they all can be active. Vertices $x$ and $y$ can be active as well. Vertex $x$ can have at most $\Delta^+ - 1$ active out-edges other than $\vec{xy}$. In the worst case all of them are marked, and to each Alice jumped twice from different in-edges of $x$, so $\vec{xy}$ has at most $2\left(\Delta^+-1\right)$ more active in-neighbors. If vertex $x$ is also marked there can be one more active in-edge of $x$ (it was activated because of the jump from one of already counted edges), and one more, from which Alice jumped to $\vec{xy}$. So $\vec{xy}$  can have at most $\Delta +3\Delta^+$ active neighbors, as claimed.
\end{proof}

Careful study of the proof and earlier observations let us formulate following conjecture.

\begin{conj}
Let G be a graph. Then $\gcol''(G)=\gcol'(G)+2$.
\end{conj}

\end{document}